\documentclass[a4paper,10pt]{scrartcl}
\usepackage[utf8]{inputenc}
\usepackage[english]{babel}
\usepackage[fixlanguage]{babelbib}
\usepackage[noadjust]{cite}
\usepackage{amssymb, amsmath, amstext, amsopn, amsthm, amscd, amsxtra, amsfonts}
\usepackage{yfonts, mathrsfs}
\usepackage{colonequals}
\usepackage{enumerate}
\usepackage{graphicx}%
\usepackage{colonequals}
\usepackage{tikz-cd}
\usetikzlibrary{cd}
\usepackage{hyperref}%
\hypersetup{
urlcolor = red,
colorlinks = true,
linkcolor = blue,
citecolor = blue,
linktocpage = true,
pdftitle = {On the vertical bisections},
pdfauthor = {Alexander Schmeding},
bookmarksopen = true,
bookmarksopenlevel = 1,
unicode = true,
hypertexnames =false
}%
\swapnumbers
\newtheorem{thm}{Theorem}[section]
\newtheorem{la}[thm]{Lemma}
\newtheorem{prop}[thm]{Proposition}

\theoremstyle{definition}
\newtheorem{defn}[thm]{Definition}
\newtheorem{exa}[thm]{Example}
\newtheorem{rem}[thm]{Remark}
\newtheorem{numba}[thm]{\!\!}

\newcommand{\mto}{\mapsto}
\newcommand{\N}{{\mathbb N}}
\newcommand{\R}{{\mathbb R}}
\newcommand{\T}{{\mathbb T}}
\newcommand{\Z}{{\mathbb Z}}

\newcommand{\Q}{{\mathbb Q}}

\newcommand{\Sph}{{\mathbb S}}

\newcommand{\cG}{{\mathcal G}}

\newcommand{\cB}{{\mathcal B}}

\newcommand{\sub}{\subseteq}

\DeclareMathOperator{\id}{id}

\newcommand{\IG}{{\mathcal IG}}

\newcommand{\toto}{\ensuremath{\nobreak\rightrightarrows\nobreak}}
\newcommand{\coloneq}{\colonequals}

\newcommand{\Bis}{\ensuremath{\operatorname{Bis}}}
\newcommand{\vBis}{\ensuremath{\operatorname{vBis}}}

\DeclareMathOperator{\Diff}{Diff}

\DeclareMathOperator{\Evol}{Evol}
\DeclareMathOperator{\evol}{evol}

\DeclareMathOperator{\ev}{ev}

\DeclareMathOperator{\one}{\mathbf{1}}

\newcommand{\Lf}{\mathbf{L}}
\begin{document}
\title{The Lie group of vertical bisections of a regular Lie groupoid}
 \author{Alexander Schmeding\footnote{TU Berlin, Germany
\href{mailto:schmeding@tu-berlin.de}{schmeding@tu-berlin.de}
}%
}
\date{}
{\let\newpage\relax\maketitle}

\begin{abstract}
In this note we construct an infinite-dimensional Lie group structure on the group of vertical bisections of a regular Lie groupoid. We then identify the Lie algebra of this group and discuss regularity properties (in the sense of Milnor) for these Lie groups. If the groupoid is locally trivial, i.e.\ a gauge groupoid, the vertical bisections coincide with the gauge group of the underlying bundle. Hence the construction recovers the well known Lie group structure of the gauge groups. To establish the Lie theoretic properties of the vertical bisections of a Lie groupoid over a non-compact base, we need to generalise the Lie theoretic treatment of Lie groups of bisections for Lie groupoids over non-compact bases.
\end{abstract}

\medskip

\textbf{MSC2010:} 
22E65 (primary);  % inf-dim Lie groups
22A22, % top groupoids including diff/Lie
58D15, % global analysis: manifolds of mappings
58H05\\[2.3mm]% pseudogroups and differentiable groupoids
%c
\textbf{Key words:}
regular Lie groupoid, Lie algebroid, infinite-dimensional Lie group, regularity of Lie groups, manifold of mappings, local triviality, gauge groupoid, gauge group, 

\tableofcontents

\section*{Introduction and statement of results} \addcontentsline{toc}{section}{Introduction and statement of results}
 Lie groupoids have found wide application in differential geometry. In particular, they can be used to formulate the symmetry of objects with bundle structure. They generalise Lie groups and their Lie theory exhibits features not present in the theory of (finite-dimensional) Lie groups (e.g.\ the integrability issue of Lie algebroids discussed in \cite{CaF,MR2795150}). 
To every (finite-dimensional) Lie groupoid one can construct an (infinite-dimensional) Lie group, the group of (smooth) bisections of the Lie groupoid \cite{Ryb,SaW,AS19}. Moreover, one can show that the geometry and representation theory of this group is closely connected to the underlying Lie groupoid. If the Lie groupoid is locally trivial (i.e.\ represents a principal fibre bundle), one can even recover the Lie groupoid from the infinite-dimensional Lie group \cite{SaW,SaW2,SaW3,AS19}.

In the present note we develop the Lie theory for the group of vertical bisections. A vertical bisection is a smooth map which is simultaneously a section for the source and the target map of the groupoid. We prove that for regular Lie groupoids, the vertical bisections form an infinite-dimensional Lie group which is an initial Lie subgroup of the group of bisections.
Before we explain this result, lets motivate the interest in groups of vertical bisections. 
Firstly, we restrict to the special case of a gauge groupoid $\text{Gauge} (P) := (P\times P / H \toto M)$ of a principal $H$-bundle $P\rightarrow M$ over a compact base $M$.\footnote{Gauge groupoids are locally trivial Lie groupoids and every locally trivial Lie groupoid arises as a gauge groupoid of a principal bundle, \cite[\S 1.3]{MAC}. Here $M$ being compact allows us to ignore some technicalities arising in the non-compact case (which is similar to the compact case (if one replaces the results in \cite{Woc} by \cite{Schue,MR1040392}).} Then the Lie group of bisections $\Bis (\text{Gauge} (P))$ is isomorphic (as an infinite-dimensional Lie group) to the group of (smooth) bundle automorphism $\text{Aut}(P)$, \cite[Example 2.16]{SaW}. Translating the vertical bisections to the bundle picture, they are identified with the automorphisms of $P$ descending to the identity on the base. Hence the group of vertical bisections $\vBis (\text{Gau} (P))$ is isomorphic to the gauge group $\text{Gau}(P)$ of the principal $H$-bundle. Thus \cite{Woc} shows that we obtain a Lie group extension
\begin{equation}\label{eq: gauge}
\begin{tikzcd}
 \vBis (\text{Gauge}(P)) \arrow[r,hookrightarrow] \arrow[d,"\cong"] & \Bis (\text{Gauge}(P)) \arrow[r,twoheadrightarrow] \arrow[d,"\cong"] & \Diff_{[P]} (M) \ar[d,"\cong"] \\
 \text{Gau}(P) \arrow[r,hookrightarrow] & \text{Aut} (P) \arrow[r,twoheadrightarrow]  & \Diff_{[P]} (M) \end{tikzcd}
\end{equation}
where $\Diff_{[P]}(M)$ is a certain open subgroup of the group $\Diff (M)$ of diffeomorphisms of $M$. Thus for locally trivial Lie groupoids our results on the vertical bisections are not new as they can be derived from the Lie theory of gauge groups, \cite{Woc,Schue}. In the present paper we seek to generalise these results to a larger class of Lie groupoids not related to principal bundles.

Secondly, (vertical) bisections are closely connected to the differential geometry of the underlying Lie groupoid. Thinking of a bisection $\sigma$ as a generalised element of the Lie groupoid, we obtain an inner automorphism, \cite[Definition 1.4.8]{MAC} and a surjective morphism onto the, in analogy to the Lie group case so-called, inner automorphisms of the groupoid
$$\pi \colon \Bis (\cG) \rightarrow \text{Inn} (\cG) \subseteq \text{Aut} (\cG), \quad \sigma \mapsto I_\sigma, \quad I_\sigma (g) := \sigma (\beta (g)) g \sigma(\alpha (g))^{-1}.$$
The vertical bisections are mapped precisely to the subgroup of inner automorphisms which preserve source and target fibres. Though the global structure of the groups $\text{Inn} (\cG), \text{Aut} (\cG)$ has to our knowledge not yet been studied, these groups are closely connected to the geometry of the Lie groupoid, \cite[Section 5]{MR3494992}.\footnote{Another example along these lines can be found in \cite[Appendix]{MR3838387}, where geometric objects such as torsion free connections are constructed using the vertical bisections of the jet groupoid.} 
Apart from the connection to inner automorphisms, we have shown in \cite{SaW2,SaW3} that for certain Lie groupoids the groupoid can be recovered from their groups of bisections. Namely, in \cite[Section 4]{SaW2} certain Lie subgroups of the bisections were crucial to this (re-)construction process. So far this process is restricted to locally trivial Lie groupoids and a generalisation would require different ingredients. One candidate which could provide additional structure usable in this (re-)construction could be the group of vertical bisections. Note however, that this group alone does not carry enough information to deal with the reconstruction of general Lie groupoids.  \medskip

In the present article, we work in the so-called Bastiani setting of infinite-dimensional analysis (i.e.\ a mapping is smooth if all iterated directional derivatives exist and are continuous, cf. references in Appendix \ref{app:calc}). Our main result is the construction of an infinite-dimensional Lie group structure on the group of vertical bisections of a regular Lie groupoid, cf.\ Theorem \ref{thm: vertBis}. Moreover, we establish some Lie theoretic properties and clarify the relation of this structure to the Lie group structure on the full group of bisections.

While we concentrate in the present paper on finite-dimensional Lie groupoids, one should be able to extend these results to Lie groupoids with infinite-dimensional space of arrows. If the base of the groupoid is compact, the theory can be adapted using results from \cite{SaW} which deal with the general case. For non-compact base manifolds it is conjectured in \cite{HS16} that similar results can be achieved.

\section{Preliminaries and the Lie group of bisections}\label{sec:prelim}
We shall write $\N=\{1,2,\ldots\}$ and $\N_0:=\N\cup\{0\}$.
Hausdorff locally convex real topological vector spaces
will be referred to as locally convex spaces.
All manifolds will be assumed to be Hausdorff spaces and if a manifold is finite-dimensional we require that it is $\sigma$-compact (for infinite-dimensional manifolds no such requirements are made).
%If $E$ and~$F$ are locally convex spaces, we let
%$\cL(E,F)$ be the space of all continuous linear mappings from~$E$ to~$F$.
%We write $\cL(E,F)_c$ and $\cL(E,F)_b$, respectively, if the topology
%of uniform convergence on compact sets (resp., bounded sets) is used on~$\cL(E,F)$.
%If $E$ is a Banach space, we write $\GL(E)\sub\cL(E)_b$ for the open subset of invertible
%operators, where $\cL(E):=\cL(E,E)$.
%We shall work in a setting of infinite-dimensional
%calculus known as Bastiani calculus and 
For manifolds $M,N$ we let $C^\infty (M,N)$ denote the set of all (Bastiani) smooth mappings from $M$ to $N$. Furthermore, we denote by $\mathcal{D} (M,TN)$ the smooth mappings $s\colon M\rightarrow TN$ such that $s=0$ off some compact set $K\subseteq M$ (i.e.\ the ``space of all smooth mappings with compact support'').  

\begin{numba}
 In the following $\mathcal{G} = (G\toto M)$ will be a (finite-dimensional) Lie groupoid with source map $\alpha$ and target map $\beta$. We denote by $\iota \colon G \rightarrow G$ the inversion and by $\one \colon M \rightarrow G$ the unit map.
\end{numba}

\begin{numba}\label{numba: bis}
 To the Lie groupoid $\cG$ we associate a group of smooth mappings, the so-called \emph{bisection group}. To this end, let 
 $$\Bis (\mathcal{G}) := \{\sigma \in C^\infty (M,G) \mid \alpha \circ \sigma = \id_M  \text{ and } \beta \circ \sigma \in \Diff (M)\},$$ 
 be the set of bisections of $\mathcal{G}$.
 The set $\Bis (\cG)$ is a group with respect to the operations 
 \begin{align*}
  \sigma \star \tau (x):= \sigma (\beta \circ \tau (x)) \tau(x) , \qquad \sigma^{-1} (x) = \iota \circ \sigma \qquad x\in M.
 \end{align*}
\end{numba}

\begin{prop}[cf.\ {\cite[Theorem 3.8]{SaW} and \cite[Proposition 1.3]{AS19}}]\label{prop:bisLie}
 Let $\mathcal{G}$ be a finite dimensional Lie groupoid, then $\Bis (\mathcal{G})$ is a submanifold of $C^\infty (M,G)$\footnote{If $M$ is non-compact, the topology on $C^\infty (M,G)$ is the so-called fine very strong topology, cf.\ \cite{HS16} and see \cite{Mic} for the construction of the manifold structure. If $M$ is compact, the fine very strong topology coincides with the familiar compact-open $C^\infty$-topology.} and this structure turns the bisections into an infinite dimensional Lie group.
\end{prop}
 
\begin{rem}\label{rem: regular}
 %Note that in the Bastiani setting the Lie algebra and advanced Lie theoretic properties of the Lie group of bisections has so far only been established for Lie groupoids with compact base \cite{SaW}. 
 %To our knowledge there is currently no citable source for the following folklore facts crucial to the Lie theoretic treatment of the bisection group for Lie groupoids over a \textbf{non-compact base}:
 %\begin{itemize}
 % \item the Lie algebra $\Lf (\Bis (\cG))$ is the Lie algebra $\Gamma_c (\Lf (\cG))$ of compactly supported sections of the Lie algebroid $\Lf (\cG)$ associated to $\cG$ with the negative of the usual Lie bracket,
 % \item the Lie group $\Bis (\cG)$ is $C^r$-regular for every $r\in \N_0 \cup \{\infty\}$, i.e.\ for every $C^r$-curve $\gamma \colon [0,1] \rightarrow \Lf (\Bis (\cG)) = \Gamma_c (\Lf (\cG))$ the initial value problem 
 % \begin{displaymath}\begin{cases}
 %    \eta'(t) = T_{\one} \rho_{\eta (t)} (\gamma(t)) \qquad \rho_{g}(h):= h\star g\\
 %   \eta (0) = \one
 %   \end{cases}
 %\end{displaymath}
 %has a unique $C^{r+1}$-solution $\Evol (\gamma) := \eta \colon [0,1] \rightarrow \Bis (\cG)$ and the map 
 %$$\evol \colon C^r ([0,1],\Gamma_c (\Lf (\cG))) \rightarrow \Bis (\cG),\quad \gamma \mapsto \Evol (\gamma)(1)$$
 %is smooth.
 %\end{itemize}
 %For Lie groupoids over a compact base $M$, proofs for these facts can be found in \cite[Section 4 and 5]{SaW}.
 %To the best of our knowledge no proof of the above exists in the literature for Lie groupoids over a non-compact base $M$. 
 Note that \cite{Ryb} establishes the Lie group structure of $\Bis (\cG)$ in the inequivalent convenient setting of global analysis (cf.\ \cite{KaM}). In general our results will imply the results from loc.cit.\ as they entail continuity of the underlying mappings, which is not automatic in the convenient setting.
\end{rem}

In the rest of this section we will prove results and discuss the necessary changes to identify the Lie algebra of bisection groups for Lie groupoids over a non-compact base. 

\begin{la}[{\cite[Corollary A.6]{AS19}}]\label{lem: ev:subm} 
 Let $\cG = (G \toto M)$ be a finite-dimensional Lie groupoid, then the evaluation mapping $\ev \colon \Bis (\cG) \times M\rightarrow G ,\ (\sigma , m) \mapsto \sigma (m)$
 is a smooth submersion.
\end{la}

\begin{la}\label{la: natact}
 Let $\cG = (G \toto M)$ be a finite-dimensional Lie groupoid, then the canonical action of the bisection group $\gamma \colon \Bis (\cG) \times G\rightarrow G , \quad (\sigma , g) \mapsto \sigma (\beta(g)).g$ is smooth.
\end{la}

\begin{proof}
 Note that we can write the action as a composition 
 $$\gamma (\sigma, g)= m(\ev (\sigma  , \beta (g)),g), \qquad \sigma \in \Bis (\cG), g \in G$$
 where $m \colon G\times_M G \rightarrow G$ denotes the multiplication map of the Lie groupoid. 
 Since $\ev$ is smooth by Lemma \ref{lem: ev:subm}, we deduce that $\gamma$ is smooth.
\end{proof}

\begin{rem}
 Having established smoothness, similar arguments as in \cite[Proposition 2.4]{SaW2} show that the restricted action 
 $\gamma_g \colon \Bis (\cG) \rightarrow \alpha^{-1} (\alpha(g)) , \sigma \mapsto \gamma (\sigma ,g)$
 is a submersion. However, we do not need this result.
\end{rem}

We adapt now the approach in \cite[Section 3]{SaW} using smoothness $\gamma$ to identify the Lie algebra of the bisection group $\Bis (\cG)$:

\begin{prop}\label{prop: LA}
 The Lie algebra of $\Bis (\cG)$ is isomorphic to the Lie algebra of smooth compactly supported sections $\Gamma_c (\Lf (\cG))$ with the negative of the usual bracket.
\end{prop}

For $M$ compact Proposition \ref{prop: LA} was established as \cite[Theorem 4.4]{SaW}. If $M$ is non-compact the function space topologies are more involved. Though the algebraic calculations carry over verbatim, the proof has to adapt smoothness arguments. 

\begin{proof}[Proof of Proposition \ref{prop: LA}]
 We assume that $M$ is not necessarily compact and $\Bis (\cG)$ is endowed with the Lie group structure from \cite[Proposition 1.3]{AS19}. According to loc.cit. the bisections are a submanifold as the preimage of the submersion $\alpha_*$ (pushforward) via $\Bis (\cG) = (\alpha_*|_{\beta_*{^-1} (\Diff (M)})^{-1} (\id_M)$. Thus we identify the Lie algebra $\Lf (\Bis (\cG) \cong T_{\one} \Bis (\cG) = \text{ker }T_{\one} \alpha_*$. Following \cite[Theorem 10.13]{Mic} the map 
 $$\Phi_{M,G} \colon TC^\infty (M,G) \rightarrow \mathcal{D}(M,TG),\ [t\mapsto \eta (t)] \mapsto (m \mapsto (t\mapsto [\eta (t)(m)]))$$
 is an isomorphism of vector bundles, where tangent vectors are identified with equivalence classes $[\eta]$ of smooth curves $\eta \colon ]-\varepsilon,\varepsilon[ \rightarrow C^\infty (M,G)$ for some $\varepsilon>0$. %(cf.\ \cite[proof of Theorem A.9]{SaW} for an explanation). 
 Up to the identification $T(\alpha_*) = (T\alpha)_*$, whence the kernel of $T\alpha_*$ is
  $$\Gamma_c (\Lf (\cG)) \coloneq \{ \gamma \in \mathcal{D} (M,TG) \mid \forall x\in M, \gamma(x) \in T_{\one_x} \alpha^{-1} (x)\}$$
 the space of compactly supported sections of the Lie algebroid $\Lf (\cG)$. Recall that the Lie bracket is induced by the bracket of right invariant vector fields on the bisections. We indicate now how to supplement the calculations in \cite[Theorem 4.4]{SaW} when the arguments involve smoothness. The following list compiles the tools and changes:
 \begin{enumerate}
 \item $\Phi_{M,G}$ restricts to an isomorphism $\varphi_\cG \colon \Lf (\Bis (\cG)) = T_{\one} \Bis (\cG)\rightarrow \Gamma_c (\Lf (\cG))$.
 \item To $X\in T_{\one} \Bis (\cG)$ associate the vector field $\overrightarrow{\varphi_\cG(X)}$ on $G$, defined via $\overrightarrow{\varphi_\cG(X)} (g) := T(R_g) (\varphi_\cG(X)(\beta (g)))$ (where $R_g$ is the right-translation in the Lie groupoid).
 \item Note that the natural action $\gamma \colon \Bis (\cG) \times G \rightarrow G$ is smooth by replacing \cite[Proposition 3.11]{SaW} with Lemma \ref{la: natact}. Thus in the proof of \cite[Proposition 4.2]{SaW} we replace \cite[Theorem 7.8 (d)]{SaW} with \cite[Lemma 10.15]{Mic} to prove that
 \item For a right invariant vector field $X^\rho$ associated to $X \in T_{\one}\Bis (\cG)$, the vector field $X^\rho \times 0$ is related to $\overrightarrow{\varphi_\cG(X)}$ via $\gamma$.
 \end{enumerate}
 Now as in \cite[Theorem 4.4]{SaW} one calculates the Lie bracket and shows that $\varphi_\cG$ is an antiisomorphism of Lie algebras.  
 \end{proof}

\section{The vertical bisections of a regular Lie groupoid}\label{sec: vertBis}

In this section we will discuss the Lie group structure of the group of vertical bisections of a Lie groupoid $\cG = (G\toto M)$. 

\begin{defn}
A \emph{vertical bisection} of $\cG$ is a bisection $\sigma \in \Bis (\cG)$ such that $\beta \circ \sigma = \id_M$. We denote the subgroup of $\Bis (\cG)$ of all vertical bisections by 
$$\vBis (\mathcal{G}) = \{ \sigma \in \Bis (\mathcal{G}) \mid \beta \circ \sigma = \id_M\}$$
\end{defn}

\begin{exa}Let $\cG$ be a Lie groupoid.
\begin{enumerate}
 \item  If $\cG$ is totally intransitive, i.e.\ source and target mapping coincide and $\cG$ is a Lie group bundle, the vertical bisections coincide with the group of bisections.
 \item If $\cG$ is a transitive Lie groupoid, i.e.\ a gauge groupoid of a principal $H$-bundle $P\rightarrow M$, the vertical bisections coincide with the gauge group of the principal bundle, cf.\ \cite[Example 2.16]{SaW}.
\end{enumerate}
\end{exa}

It is not hard to see that the subgroup $\vBis (\cG)$ is a normal subgroup of $\Bis (\cG)$, \cite[Proposition 1.1.2.]{MR3494992}. As the pushforward $\beta_* \colon C^\infty (M,G) \rightarrow C^\infty (M,M) , f \mapsto \beta \circ f$ is smooth (whence in particular continuous), $\vBis(\cG)$ is a closed subgroup of $\Bis(\cG)$.
Unfortunately, this does not entail that $\vBis (\cG)$ is a Lie subgroup of $\Bis (\cG)$, as $\Bis (\cG)$ is an infinite dimensional Lie group.

The vertical bisections are exactly the bisections which take their values in the isotropy subgroupoid $\IG$ (cf.\ Appendix \ref{app:calc}) of $\cG$, i.e.\ 
$$\vBis (\cG) = \{ \sigma \in \Bis (\cG) \mid \sigma (M) \subseteq \IG\}.$$ 
If $\IG$ would be a Lie subgroupoid (which it in general is not, cf.\ \ref{exa:notLie}), we could identify the vertical bisections as the group of (smooth) bisections of the isotropy subgroupoid. However, there is a large class of Lie groupoids for which at least the connected identity subgroupoid $\IG^\circ$ of the isotropy groupoid $\IG$ (cf.\ Appendix \ref{app:calc}) is an embedded Lie subgroupoid.

\begin{defn}
 A Lie groupoid $\cG = (G\toto M)$ is called \emph{regular Lie groupoid} if the Lie groupoid anchor
 $$(\alpha ,\beta ) \colon G \rightarrow M\times M , \quad g \mapsto (\alpha(g),\beta(g))$$
 is a mapping of constant rank.
\end{defn}

\begin{rem}
 Many important classes of Lie groupoids, such as foliation groupoids of regular foliations, transitive groupoids and locally trivial groupoids are regular groupoids, cf.\ \cite{Wei02, Moer} for more information.
 The regularity condition on the anchor $(\alpha,\beta) \colon G \rightarrow M\times M$ is equivalent to requiring that the anchor $\rho \colon \Lf (\cG) \rightarrow TM$ of the associated Lie algebroid is of constant rank, cf.\ \cite{Wei02}.
\end{rem}

\begin{la}[{\cite[Proposition 2.5]{Moer}}]\label{la:Moerdijk}
 Let $\cG$ be a regular Lie groupoid, then the connected identity subgroupoid $\IG^\circ$ is an embedded normal Lie subgroupoid of $\cG$. Its associated Lie algebroid $\Lf(\IG^\circ)$ is the isotropy subalgebroid $\mathcal{I}\Lf(\cG)$ of $\Lf (\cG)$.\footnote{Recall that the isotropy subalgebroid $\mathcal{I}\Lf(\cG)$ is given fibre-wise as the kernel $\text{ker} (\rho_x)$, or equivalently as the Lie algebra $\Lf (\alpha^{-1}(x)\cap \beta^{-1}(x))$ of the isotropy subgroup at $x$, cf.\ \cite[2.2]{MR2795150}.}
\end{la}

Let us remark here that even for regular Lie groupoids, the subgroupoid $\IG$ is in general not an embedded Lie subgroupoid, as the following example shows:

\begin{exa}\label{exa:joao}
  Let $\T := \R^2 / \Z^2$ be the $2$ dimensional torus. Consider the action of $(\R,+)$ on $\T \times ]-1,1[$ via
 $$\lambda ([x,y], \varepsilon) := ([x+\lambda, y+\lambda \varepsilon], \varepsilon), \qquad \lambda \in \R , [x,y] \in \T \text{ and } \varepsilon \in ]-1,1[.$$
 The associated action groupoid $A$ is regular as all orbits are diffeomorphic either to circles or lines. Furthermore, the isotropy at a given point $([x,y],\varepsilon) \in A$ is either a copy of $\Z$ in $\R \times \{([x,y],\varepsilon)\}$ if $\varepsilon$ is rational or a singleton for $\varepsilon$ irrational.
 As all the points in the same orbit have isotropy of the same type, this implies that the isotropy subgroupoid would have to be at least a one dimensional submanifold if it were an embedded submanifold. 
 However, this implies that $IA$ can not be an embedded Lie groupoid as for example there is no neighborhood of the point $(1,[x,y],0) \in IA$ which is diffeomorphic to a non-trivial euclidean space (due to the trivial isotropy groups of the points $([x,y],\varepsilon)$ for $\varepsilon \in \R\setminus \Q$.  
\end{exa}

Thus we can not leverage in the following constructions a smooth structure on the isotropy groupoid. Note that by restricting ourselves to the smaller class of locally trivial Lie groupoids, such a structure would be available, as then $\IG$ is indeed an embedded submanifold, \cite[Proposition 1.17]{MAC87}. 
Instead we will now describe a construction of a Lie group structure on $\vBis (\cG)$ which works for every regular Lie groupoid. To this end we leverage that $\IG^\circ$ is an embedded submanifold and consider an auxiliary group 
$$\vBis^\circ (\cG) := \{\sigma \in \Bis (\cG) \mid \sigma (M) \subseteq \IG^\circ\}$$
Since $\IG^\circ$ is an embedded Lie subgroupoid, we can test smoothness of a bisection taking its values in $\IG^\circ$ with respect to the submanifold structure \ref{into-sub}. Hence, we obtain the following. 

\begin{prop}\label{prop: vBiscirc}
 Let $\cG$ be a regular Lie groupoid. Then $\vBis^\circ(\cG) \subseteq C^\infty (M,G)$ is a submanifold and this structure turns $\vBis^\circ(\cG)$ into an infinite-dimensional Lie group which is isomorphic to $\Bis (\IG^\circ)$. Moreover, this Lie group satisfies the following:
 \begin{enumerate}
  \item The Lie algebra $\Lf (\vBis^\circ (\cG))$ is isomorphic to the Lie algebra $\Gamma(\mathcal{I}\Lf (\cG))$ of smooth sections  of the isotropy algebroid with the negative of the usual bracket.
  \item The inclusion $\iota_{\Bis} \colon \vBis^\circ (\cG) \rightarrow \Bis (\cG)$ turns $\vBis (\cG)$ into an initial Lie subgroup of $\Bis (\cG)$.
 \end{enumerate}
\end{prop}

\begin{proof}
 Due to Lemma \ref{la:Moerdijk}, we can consider the embedded subgroupoid $\IG^\circ \subseteq \cG$. Denoting by $I \colon \IG^\circ \rightarrow \cG$ the associated embedding, the mapping 
 $I_* \colon C^\infty (M,\IG^\circ) \rightarrow C^\infty (M,G), f\mapsto I\circ f$ is a smooth embedding, realising $C^\infty (M, \IG^\circ)$ as a split submanifold of $C^\infty (M,G)$, cf.\ \cite[Proposition 10.8]{Mic}. 
 We recall from \ref{numba: bis} that $\Bis (\IG^\circ)$ is a submanifold of $C^\infty (M,\IG^\circ)$ and this structure turns it into a Lie group. Now the canonical identification $I_* (\Bis (\IG^\circ)) = \vBis^\circ (\cG)$ shows that the Lie group $\vBis^\circ (\cG) \cong \Bis (\IG^\circ)$ can be identified as a submanifold of $C^\infty (M,G)$ (cf.\ \cite[Lemma 1.4]{SUB}).
 We establish now the properties claimed in the proposition:
 \begin{enumerate}
  \item  The isomorphism $\Bis (\IG^\circ) \cong \vBis^\circ (\cG)$, identifies the Lie algebra of $\vBis^\circ (\cG)$ with $\Lf (\Bis (\IG^\circ)$. Since the isotropy algebroid is the Lie algebroid of $\IG^\circ$, Proposition \ref{prop: LA} shows that we obtain the Lie algebra claimed in the statement of the proposition.
  \item To see that $\vBis^\circ (\cG)$ is an initial subgroup of $\Bis (\cG)$, note first that $\vBis^\circ (\cG)$ and $\Bis (\cG)$ are both submanifolds of $C^\infty (M,G)$. Thus a mapping $f \colon M \rightarrow C^\infty (M,G)$ from a $C^k$-manifold which takes its image in $H \in \{\vBis^\circ (\cG), \Bis (\cG)\}$ is of class $C^k$ if and only if it is a $C^k$-mapping into $H$. Now the inclusion $\iota_{\Bis}$ is an injective group morphism. Composing $\iota_{\Bis}$ with the inclusion $\Bis (\cG) \subseteq C^\infty(M,G)$, we obtain the (smooth) inclusion $\vBis^{\circ}(\cG) \subseteq C^\infty (M,G)$, whence $\iota_{\Bis}$ is an injective Lie group morphism. It is easy to see that $\Lf (\iota_{\Bis}) \colon \Lf (\vBis^\circ (\cG)) \rightarrow \Lf (\Bis (\cG))$ is injective, as it is, up to an identification, just the inclusion of subspaces.
 Hence we consider a $C^k$-map $f \colon N \rightarrow \Bis (\cG)$ taking its image in $\vBis^\circ (\cG)$. Then $\iota_{\Bis}^{-1} \circ f$ is a $C^k$-map into $\vBis^\circ (\cG)$, as we can identify it with the $C^k$-map $f\colon M \rightarrow \Bis (\cG) \subseteq C^\infty (M,G)$. 
 \end{enumerate}
\end{proof}

Note that the above proof gives no information about $\vBis^\circ (\cG)$ being a Lie subgroup of $\Bis (\cG)$ in the traditional sense (i.e.\ being an embedded submanifold). However, we will now leverage the structure on  $\vBis^\circ (\cG)$ to construct a Lie group structure on $\vBis (\cG)$ via the construction principle Proposition \ref{prop:Bourbaki}.

\begin{thm}\label{thm: vertBis}
 Let $\cG$ be a regular Lie groupoid, then the group of vertical bisections $\vBis (\cG)$ is an infinite-dimensional Lie group with $\Lf (\vBis (\cG)) = \Lf (\vBis^{\circ} (\cG))$. 
 With respect to this structure, the vertical bisections form an initial Lie subgroup of $\Bis (\cG)$. Moreover, $\vBis^{\circ} (\cG)$ becomes an open subgroup of $\vBis (\cG)$.
\end{thm}

\begin{proof}
 Apply the construction principle Proposition \ref{prop:Bourbaki}: Set $U=V=\vBis^\circ(\cG)$ and observe that part (a) just yields the Lie group structure on $\vBis^\circ (\cG)$ from Proposition \ref{prop: vBiscirc}. We wish now to apply part (b) of Proposition \ref{prop:Bourbaki} to obtain a Lie group structure on $\vBis (\cG)$. 
 
 To this end, we observe that $\IG^\circ$ forms a normal Lie subgroupoid of $\cG$, \ref{la:Moerdijk}. Let now $\sigma, \tau \in \vBis (\cG)$, then one directly verifies from the formula for multiplication and inversion in the bisection group that 
 $$c_\tau (\sigma) (x) := \tau \star \sigma \star \tau^{-1} (x) = \tau (x) \sigma (x) (\tau (x))^{-1},$$
 where $(\tau (x))^{-1}$ denotes the inverse of $\tau(x)$ in the Lie group $\alpha^{-1}(x) \cap \beta^{-1}(x)$.
 If $\sigma \in \vBis^\circ (\cG)$, we deduce that $c_\tau (\sigma)(x)$ stays in the normal subgroup $\alpha^{-1}(x) \cap \beta^{-1}(x) \cap \IG^\circ$, whence $\vBis^\circ (\cG)$ is a normal subgroup of $\vBis (\cG)$.
 Hence we set $W:=\vBis^\circ(\cG)$ and have to prove that $c_\tau \colon W \rightarrow W$ is smooth for every $\tau \in \vBis (\cG)$. 
 As $\Bis (\cG)$ is a Lie group, $C_\tau \colon \Bis (\cG) \rightarrow \Bis (\cG), \delta \mapsto \tau \star \delta \star \tau^{-1}$ is smooth. In Proposition \ref{prop: vBiscirc} we have seen that the inclusion $\iota_{\Bis}\colon \vBis^\circ (\cG) \rightarrow \Bis (\cG)$ turns $\vBis^\circ (\cG)$ into an initial Lie subgroup of $\Bis (\cG)$. Combining these observations we conclude from Proposition \ref{prop:Bourbaki} (b) that $\vBis (\cG)$ is a Lie group as $c_\tau = \iota_{\Bis}^{-1}\circ C_\tau \circ \iota_{\Bis}$ is smooth for every $\tau \in \vBis (\cG)$.
 
 Since $\vBis^\circ(\cG)$ is an open Lie subgroup of $\vBis (\cG)$ it is clear that the Lie algebras of both groups coincide. To see that $\vBis (\cG)$ is an initial Lie subgroup of $\Bis (\cG)$ we observe that the inclusion $I_{\Bis} \colon \vBis (\cG) \rightarrow \Bis (\cG)$ is an injective morphism of Lie groups with $\Lf (I_{\Bis})$ injective since $\iota_{\Bis}$ is such a morphism. Let now $f\colon N \rightarrow \Bis (\cG)$ be a $C^k$-map with image in $\vBis (\cG)$. It suffices to check the $C^k$-property on every (open) connected component of $\vBis (\cG)$. Without loss of generality, we may thus assume that $f$ takes its image in a component $C \subseteq \vBis (\cG)$ such that for some $g\in \vBis (\cG)$ we have $g^{-1} \star C \subseteq \vBis^\circ (\cG)$. Denoting left translation by an element $\ell$ of a Lie group $L$ by $\lambda_\ell$, we see that $$I_{\Bis}^{-1} \circ f = \lambda_{g}^{\vBis (\cG)} \circ\iota_{\text{Bis}}^{-1} \circ \lambda_{g^{-1}}^{\Bis}f,$$
 whence $I_{\Bis}^{-1} \circ f$ is a $C^k$-map. 
\end{proof}

We are now in a position to establish regularity (in the sense of Milnor) for the group of vertical bisections. Recall that a Lie group $H$ is $C^r$-regular if for every $C^r$-curve $\gamma \colon [0,1] \rightarrow \Lf (H))$ the initial value problem 
  \begin{equation}\begin{cases}
     \eta'(t) =  T_{\one} \rho_{\eta (t)} (\gamma(t)) \qquad \rho_{g}(h):= h g\\
     \eta (0) = \one
    \end{cases}
 \end{equation}
 has a unique $C^{r+1}$-solution $\Evol (\gamma) := \eta \colon [0,1] \rightarrow H$ and the evolution map \\
 $\evol \colon C^r ([0,1],\Lf (H)) ) \rightarrow H, \gamma \mapsto \Evol (\gamma)(1)$ is smooth. To employ advanced techniques in infinite-dimensional Lie theory, one needs to require regularity of the Lie groups involved, cf.\ \cite{hg2015}.

\begin{prop}\label{prop:reg}
The Lie group $\vBis (\cG )$ is $C^r$-regular for every $r \in \N_0\cup \{\infty\}$, whenever $M$ is compact or $\cG$ is a transitive Lie groupoid. 
\end{prop}

\begin{proof}
 We distinguish two cases:\\
 \textbf{Case 1: $M$ is compact.} Since $\vBis^\circ (\cG) \cong \Bis (\IG^\circ)$ is an open subgroup of $\vBis(\cG)$, we see that $\vBis (\cG)$ is $C^r$-regular, if and only if $\Bis (\IG^\circ)$ is $C^r$-regular. However, the $C^r$-regularity of $\Bis (\IG^\circ)$ was established in \cite[Theorem 5.5]{SaW}.\\
 \textbf{Case 2: $\cG$ is transitive.} In this case $\cG$ can be identified as a gauge groupoid of a principal $H$-bundle $P \rightarrow M$. As explained in the introduction \eqref{eq: gauge}, we can identify the compactly supported vertical bisections $\vBis (\cG)$ with the group of compactly supported gauge transformations $\text{Gau}_c (P)$. However, $\text{Gau}_c (P)$ (and thus also $\vBis (\cG)$ is $C^r$-regular for every $r\in \N_0 \cup\{\infty\}$ by a combination of \cite[Theorem A and Corollary 8.3]{MEAS}. 
\end{proof}

We expect Proposition \ref{prop:reg} to hold for all vertical bisection groups, as all bisection groups of finite-dimensional Lie groupoids are expected to be regular. For groupoids over a non-compact base, these results require mild generalisations of the results obtained in \cite{SaW,SaW2}. The main issue is that the function space topologies are much more involved in this case, see \cite{HS16,Mic}. Working around this would require extensive localisation arguments (on a cover of compact sets) which poses no conceptual problem, but would lead quite far away from the main line of reasoning. Thus we have not established the result in full generality.

Note that from the construction it is not clear whether $\vBis (\cG)$ is a Lie subgroup of $\Bis (\cG)$ and not even if $\vBis (\cG)$ is a submanifold of $C^\infty (M,G)$. Up to this point we can only obtain the following:

\begin{la}
 The Lie group topology of $\vBis (\cG)$ from Theorem \ref{thm: vertBis} is the subspace topology induced by $C^\infty (M,G)$. Moreover, a mapping $f\colon M \rightarrow C^\infty (M,G)$ whose image is contained in $\vBis (\cG)$ is of class $C^k$ if and only if it is $C^k$ as a mapping into $\vBis (\cG)$.
  \end{la}

\begin{proof}
 The statement about $f\colon N \rightarrow C^\infty (M,G)$ follows from $\vBis (\cG)$ being an initial Lie subgroup of $\Bis (\cG)$, Theorem \ref{thm: vertBis}, and the fact that $\Bis (\cG)$ is a submanifold of $C^\infty (M,G)$. 
 
  To see that the topology on $\vBis (\cG)$ coincides with the subspace topology induced by the inclusion $\vBis (\cG) \subseteq C^\infty (M,G)$, recall that the Lie group topology on $\Bis (\cG)$ is the subspace topology induced by $C^\infty (M,G)$ (cf.\ \cite[Proposition 1.3]{AS19}. Further $\vBis^\circ (\cG)$ carries the subspace topology of $C^\infty (M,G)$, Proposition \ref{prop: vBiscirc}, whence the subspace topology induced by $\Bis (\cG)$. Now for $\tau \in \vBis (\cG)$ the left translations $\lambda_\tau \colon \Bis (\cG) \rightarrow \Bis (\cG) , \sigma \mapsto \tau \star \sigma$ is a homeomorphism mapping $\vBis (\cG)$ to $\vBis (\cG)$. We conclude that every component of $\vBis (\cG)$ carries the subspace topology induced by the inclusion $\vBis (\cG) \subseteq \Bis(\cG) \subseteq C^\infty (M,G)$.
\end{proof}

The vertical bisections encode isotropy information of the underlying Lie groupoid. If the Lie groupoid contains only 'small' isotropy groups, the subgroup of vertical bisections is a very small subgroup as the next example shows.

\begin{exa}
 Consider a proper \'{e}tale Lie groupoid\footnote{Proper \'{e}tale Lie groupoid are also known as ``orbifold groupoids`` as they represent orbifolds. See \cite{MaP} for more information.} $\cG$, i.e.\ a Lie groupoid with proper anchor map such that $\alpha,\beta$ are local diffeomorphisms. Then the isotropy subgroup $\cG_x :=\alpha^{-1} (x) \cap \beta^{-1} (x)$ is discrete. 
 Hence $\IG^\circ = \one (M) \subseteq G$ and we have $\vBis^{\circ} (\cG) = \{\one\}$. From the construction of the Lie group $\vBis (\cG)$ via Proposition \ref{prop:Bourbaki} in Theorem \ref{thm: vertBis} it is then clear that $\vBis (\cG)$ is a discrete Lie group. 
\end{exa}

\begin{rem}
 Albeit the smooth structure on $\vBis (\cG)$ is constructed by translating the smooth structure of $\vBis^\circ (\cG)$ along the diffeomorphisms $\lambda_\tau$ it is still not clear whether $\vBis (\cG)$ is a submanifold of $C^\infty (M,G)$ as it is not clear that the connected components of $\vBis^\circ (\cG)$ can be separated from each other in the topology of $C^\infty(M,G)$. 
\end{rem}

However, if the isotropy subgroupoid is an embedded Lie subgroupoid, we can indeed obtain $\vBis (\cG)$ as a submanifold of $C^\infty (M,G)$.

\begin{prop}
 Let $\cG = (G\toto M)$ be a regular Lie groupoid such that the isotropy subgroupoid $\IG$ is an embedded Lie subgroupoid, e.g.\ if $\cG$ is a locally transitive Lie groupoid. Then the Lie group $\vBis (\cG)$ from Theorem \ref{thm: vertBis} is a submanifold of $C^\infty (M,G)$.
\end{prop}

\begin{proof}
 Since $\IG$ is an embedded submanifold of $G$, we can argue as in the proof of Proposition \ref{prop: vBiscirc} to see that $\vBis (\cG) \cong \Bis (\IG) \subseteq C^\infty (M,G)$ is a submanifold. Moreover, the manifold structure turns $\vBis (\cG)$ into a Lie group. To distinguish the new Lie group structure from the one inherited from Theorem \ref{thm: vertBis}, we write $\widetilde{\vBis}(\cG)$ for this Lie group.
 As $\widetilde{\vBis} (\cG)$ is an embedded submanifold of $C^\infty (M,G)$, we can argue as in the proof of (in particular part (b)) of Proposition \ref{prop: vBiscirc} to see that $\widetilde{\vBis} (\cG)$ is an initial Lie subgroup of $\Bis (\cG)$. As any subgroup of a given Lie subgroup carries at most one structure as an initial Lie subgroup \cite[Lemma II.6.2 ]{Nee}, we have $\vBis (\cG) = \widetilde{\vBis} (\cG)$ as infinite-dimensional Lie groups. 

 Finally, assume that $\cG$ is locally transitive, i.e.\ its anchor $(\alpha, \beta) \colon G \rightarrow M\times M$ is a submersion. Then $\IG = (\alpha, \beta)^{-1} (\Delta M)$ is an embedded submanifold of $G$, where $\Delta M$ is the diagonal embedded in $M\times M$.
\end{proof}

 To provide a different geometric interpretation of the vertical bisections we deviate from our usual convention and consider an infinite-dimensional Lie groupoid. In \cite[Definition 2.1]{SaW2} we have constructed an (infinite-dimensional) action groupoid from the natural action of $\Bis (\cG)$ on $M$:
 
\begin{numba}
 The group $\Bis(\cG)$ acts on $M$, via the natural action of $\Diff (M)$ on $M$ composed with the morphism $\beta_* \colon \Bis(\cG)\to \Diff(M), \sigma \mapsto \beta \circ \sigma$. We can thus define an action Lie groupoid $\cB(\cG)\coloneq \Bis(\cG)\ltimes M$, with source and target projections defined by
 $\alpha_{\cB}(\sigma,m)=m$ and $\beta_{\cB}(\sigma,m)=\beta(\sigma(m))$.
 The multiplication on $\cB(\cG)$ is defined by
 \begin{equation*}
  (\sigma,\beta_{\cG}(\tau(m)))\cdot (\tau,m)\coloneq(\sigma \star \tau,m).
 \end{equation*}
 The Lie groupoid $\cB (\cG)$ plays a crucial r\^{o}le in the reconstruction of the Lie groupoid $\cG$ from its group of bisections (see \cite[Section 2]{SaW} for more information on this process). 
 From the definition of the action Lie groupoid, one immediately obtains that 
 the vertical bisections determine the isotropy subgroupoid, i.e.\ $\mathcal{IB}(\cG) = \vBis (\cG) \ltimes M$.
\end{numba}

\begin{numba}As a final remark, the group $\vBis (\cG)$ can be generalised (as observed by H.\ Amiri) by considering $\{f\in C^\infty (G,G) \mid \alpha \circ f = \alpha = \beta \circ f, x\mapsto xf(x) \in \Diff (G)\}$. This set turns out to be a subgroup of the group $S_{\cG}(\alpha)$ from \cite{AS17} (which generalises $\Bis (\cG)$). Similar techniques to the ones in the present paper can be used to turn the generalised group into a Lie group.
\end{numba}

\paragraph{Acknowledgments} 
The author is indepted to J.N.\ Mestre for pointing out Example \ref{exa:joao} and useful literature he was unaware of.\footnote{cf.\ \url{https://mathoverflow.net/questions/329939/isotropy-subgroupoid-of-a-regular-lie-groupoid}} 
Furthermore, he thanks S.\ Paycha and H.\ Amiri for interesting discussions concerning the subject of this work. Furthermore, he thanks the anonymous referee for insightful comments which helped improve the article.
%
%\newpage
\appendix
\section{Infinite-dimensional calculus, Lie groups and Lie groupoids}\label{app:calc}

In this appendix we recall some basic facts on the infinite-dimensional analysis used throughout the text. 
For more information we refer the reader to~\cite{Bas}, and the generalizations thereof (see \cite{GaN} and \cite{AaS},
also \cite{RES,Ham,Mic}, and~\cite{Mil}). As already remarked, we are working in the Bastiani calculus, where $f$ is $C^k$-map if all iterated directional derivatives up to order $k$ exist and are continuous.

\begin{defn}
For a smooth map $f\colon M\rightarrow N$ between manifolds, we say 
(see \cite{Ham,SUB}) that $f$ is
\begin{enumerate}
\item a \emph{submersion} if for each $x\in M$ we can choose a chart $\psi$
of~$M$ around~$x$ and a chart $\phi$ of~$N$
around $f(x)$ such that
$\phi\circ f\circ~\psi^{-1}$
is the restriction of a continuous linear map with continuous linear right inverse,
\item an \emph{immersion} if for every $x\in M$ there are charts such that we can always achieve that $\phi\circ f\circ \psi^{-1}$ is the restriction of a continuous linear map admitting a continuous linear left inverse.
\item an \emph{embedding} if $f$ is an immersion and a topological embedding.
\end{enumerate}
\end{defn}
Note that the above definitions (of submersions etc.) are adapted to the infinite-dimensional setting we are working in. In general they are not equivalent to the usual characterisations known from the finite-dimensional setting. See e.g.\ \cite{SUB}.% for a discussion. 
\begin{numba}[Submanifolds]
Let $M$ be a $C^k$-manifold.
A subset $N\sub M$ is called a \emph{submanifold}
if, for each $x\in N$, there exists a chart $\phi\colon U_\phi\to V_\phi\sub E_\phi$
of~$M$ with $x\in U_\phi$ and a closed vector subspace $F\sub E_\phi$
such that $\phi(U_\phi\cap N)=V_\phi\cap F$. Then~$N$ is a $C^k$-manifold
in the induced topology, using the charts $\phi|_{U_\phi\cap N}\colon U_\phi\cap N\to V_\phi\cap F$.
\end{numba}
\begin{numba}\label{into-sub}
If $N$ is a submanifold of a smooth manifold~$M$ and $f\colon L\to M$ a map on a smooth manifold~$L$
such that $f(L)\sub N$, then $f$ is smooth if and only if its corestriction
$f|^N\colon L\to N$ is smooth for the smooth manifold structure induced on~$N$.
\end{numba}

\subsection*{Lie groups and Lie groupoids}

We follow here \cite{Mil,Nee,GaN} for the basic theory concerning infinite-dimensional Lie groups (modeled on locally convex spaces) and \cite{MAC,Mein17} for (finite-dimensional) Lie groupoids and Lie algebroids.\footnote{The concept of infinite-dimensional Lie groupoid is clear, cf.\ \cite{SaW,SaW2} and \cite{BaGaJaP} tough not needed for most of the text.}

%\begin{defn}
% A Lie group (modeled on a locally convex space) is a group $G$ carrying the structure of a smooth manifold in the sense of Appendix \ref{app:calc} such that the group operations are smooth with respect to the manifold structure.
%\end{defn}

\begin{rem}
For a Lie group $G$ we write $\one$ for the unit element. As in the finite dimensional setting one can associate to $G$ a Lie algebra $\Lf (G) \cong T_{\one} G$ whose Lie bracket is constructed from the Lie bracket of left invariant vector fields on $G$. 
\end{rem}

\begin{defn}
 Let $G,H$ be (infinite-dimensional) Lie groups and $\varphi \colon H \rightarrow G$ be an injective morphism of Lie groups. We call $H$ an \emph{initial Lie subgroup} if the induced Lie algebra morphism $\Lf (\varphi) \colon \Lf (H) \rightarrow \Lf (G)$ is injective, and for each $C^k$-map $f\colon N\rightarrow G, (k \in \N \cup \{\infty\})$ from a $C^k$-manifold $N$ to $G$ whose image $\text{im}(f)$ is containted in $H$, the corresponding map $\varphi^{-1} \circ f \colon N \rightarrow H$ is $C^k$.
\end{defn}

We furthermore need the following construction principle for Lie groups, whose proof for manifolds modeled on Banach spaces (which generalises verbatim to our more general setting) can be found in \cite[III. \S1.9, Proposition 18]{Bou}.

\begin{prop}\label{prop:Bourbaki}
 Let $G$ be a group and $U,V$ be subsets of $G$ such that $\one \in V = V^{-1}$ and $V V \subseteq U$. Suppose that $U$ is equipped with a smooth manifold structure such that $V$ is open in $U$ which turns the inversion $\iota \colon V\rightarrow V\subseteq U$ and the multiplication $\mu \colon V \times V \rightarrow U$ -- induced by the group -- into smooth maps. Then the following holds:
 \begin{enumerate}
  \item There is a unique smooth manifold structure on the subgroup $G_0 := \langle V\rangle$ of $G$ generated by $V$ such that $G_0$ becomes a Lie group, $V$ is open in $G_0$, and such that $U$ and $G_0$ induce the same smooth manifold structure on $V$,
  \item Assume that for each $g$ in a generating set of $G$, there is an open identity neighborhood $W \subseteq U$ such that $gWg^{-1} \subseteq U$ and $c_g \colon W \rightarrow U, h \mapsto ghg^{-1}$ is smooth. Then there is a unique smooth manifold structure on $G$ turning $G$ into a Lie group such that $V$ is open in $G$ and both $G$ and $U$ induce the same smooth manifold structure on the open subset $V$.
 \end{enumerate}
\end{prop}
%
%We now fix the conventions concerning Lie groupoids.
% Note that in this article we will restrict ourselves to finite dimensional Lie groupoids albeit the definitions and basic approach make sense also for groupoids with finite dimensional base and infinite dimensional space of arrows. 

\begin{numba}
A groupoid~$\mathcal{G} = (G\toto M)$, with source map $\alpha\colon G\to M$ and target map $\beta\colon G\to M$ is a \emph{Lie groupoid}, if the following holds: $G$ and~$M$ are smooth manifolds, $\alpha$ and~$\beta$ are $C^\infty$-submersions and the multiplication map $G^{(2)}\to G$, the inversion map $G\to G$
and the identity-assigning map $M\to G, x\mto \mathbf{1}_x$ are smooth.
Recall that one can associate to every Lie groupoid $\cG$ a Lie algebroid which we denote by $\Lf (\cG)$.
\end{numba}

\begin{defn}
Let $F \colon \mathcal{H} \rightarrow \mathcal{G}$ be a morphism of Lie groupoids. We call $\mathcal{H}$ 
\begin{itemize}
\item \emph{immersed subgroupoid} of $\mathcal{G}$ if $F$ and the induced map on the base are injective immersions.
\item \emph{embedded subgroupoid} of $\mathcal{G}$ if $F$ and the induced map on the base are embeddings. 
\end{itemize}
\end{defn}

\begin{defn}
 For a Lie groupoid $\cG = (G\toto M)$ we define the following:
 \begin{itemize}
 \item Denote for $m\in M$ by $C_m$ the connected component of $\one_m$ in $\alpha^{-1} (m)$. Then% we define the subset 
  $$C(\cG) := \bigcup_{n\in M} C_m$$
  By \cite[Proposition 1.5.1]{MAC} we obtain a wide Lie subgroupoid $C(\cG) \toto M$ of $\cG$, called the \emph{identity-component subgroupoid} of $\cG$.
  \item the \emph{isotropy subgroupoid} $\IG := \{ g \in G \mid \alpha (g) = \beta (g)\},$
        and the identity component subgroupoid of the isotropy groupoid $\IG^\circ := C(\IG)$.
 \end{itemize}
\end{defn}

Endowed with the subspace topology, $\IG$ (and also $\IG^\circ$) are topological bundles of Lie groups.  
In general the isotropy subgroupoid is not a Lie subgroupoid:% as the next (well known) example shows:

\begin{exa}\label{exa:notLie}
 Let $\mathcal{A} = (\mathbb{S}^1 \times \R^2 \toto \R^2$ be the action groupoid associated to the canonical action of the circle group $\Sph^1$ on $\R^2$ via rotation. 
 Then $$\mathcal{IA} = \Sph^1 \times \{0\} \cup \bigsqcup_{x\in \R^2 \setminus \{0\}} \{1\} \times \{x\} \subseteq \Sph^1 \times \R^2$$ is not a submanifold of $\Sph^1 \times \R^2$ and thus $\mathcal{IA}$ can not be a Lie groupoid of $\mathcal{A}$.
\end{exa}

\addcontentsline{toc}{section}{References}

\bibliography{vBis}

\end{document}